\documentclass[12pt]{amsart}
\usepackage[all]{xy}
\usepackage[parfill]{parskip}

\usepackage{verbatim}
\usepackage{color}

\usepackage{amsmath, amscd, graphicx, latexsym, hyperref, times}
\usepackage{color,graphicx}
\usepackage[abs]{overpic}
\usepackage{tikz}
\usepackage{tikz-cd}
\usetikzlibrary{arrows, patterns}

\oddsidemargin .25in \theoremstyle{plain}
\newtheorem{theorem}{Theorem}
\newtheorem{thmx}{Theorem}

\newtheorem{corollary}{Corollary}

\newtheorem{remark}{Remark}

\numberwithin{equation}{section}
\numberwithin{lemma}{section}
\numberwithin{proposition}{section}
\numberwithin{corollary}{section}
\numberwithin{remark}{section}

\begin{document}

\title{A note on Hang-Wang's hemisphere rigidity theorem}

\author{Mijia Lai}
\address{School of Mathematical Sciences,
Shanghai Jiao Tong University}
\email{laimijia@sjtu.edu.cn}
\thanks{Lai's research is supported in part by National Natural Science Foundation of China No. 11871331.}
\date{}

\begin{abstract}
Let $(M,g)$ be a compact manifold with boundary and $Ric_g\geq (n-1)g$, Hang and Wang proved
that $(M,g)$ is isometric to the standard hemisphere if $\partial M$ is convex and isometric to $\mathbb{S}^{n-1}(1)$.
We prove some rigidity theorems when $\partial M $ is isometric to a product manifold where one factor is the standard sphere.
\end{abstract}
\maketitle
\section{Introduction}
Rigidity phenomenon of manifolds under various curvature conditions is a very interesting and important subject in differential geometry.
A prominent example is the rigidity part of the positive mass theorem proved by Schoen-Yau~\cite{SY} and Witten~\cite{W}, which generates
enormous study on rigidity phenomena with assumptions on the scalar curvature. In particular, by Bartnik's version of the positive mass
theorem, any metric on $\mathbb{R}^n$ with nonnegative scalar curvature, which agrees with the standard Euclidean metric
outside a compact set, must be flat. For further developments, we refer the reader to the survey~\cite{B} and
references therein.

In an attempt to tackle the Min-Oo conjecture(remarkably disproved later by Brendle-Marques-Neves~\cite{BMN})
, Hang and Wang~\cite{HW} proved an interesting rigidity theorem
for manifolds with boundary under the assumption of positive Ricci curvature.

\begin{thmx}[Hang-Wang] \label{TA}
Let $(M, g)$ be a compact manifold with boundary and suppose
\begin{itemize}
\item $Ric\geq (n-1)g$;
\item $\partial M $ is isometric to $\mathbb{S}^{n-1}(1)$;
\item $\partial M$ is convex, i.e., the second fundamental form $h\geq 0$.
\end{itemize}
Then $(M, g)$ is isometric to the standard hemisphere $\mathbb{S}^{n}_{+}(1)$.
\end{thmx}

The above three curvature assumptions are reminiscent of Serrin's overdetermined problem~\cite{S}. Roughly speaking,
the Ricci curvature can be viewed as the Laplacian of the metric. The boundary metric and the second fundamental form can be viewed
as Dirichlet and Neumann boundary conditions for the metric respectively.

In this short note, we generalize Hang-Wang's rigidity theorem in three settings depending on
the sign of the Ricci curvature lower bound. We assume the boundary is isometric to a product manifold
with one of the factors being isometric to a round sphere. The proofs are based on
the proof of Theorem~\ref{TA} with a new twist: the Obata equation with Robin boundary condition, which
has been carefully studied by Chen-Lai-Wang ~\cite{CLW}. In our proofs, we first obtain functions on $M$ satisfying Obata equations following the method of Hang-Wang.
Then we gain precise information on the second fundamental form of the inward equidistance hypersurfaces. The precise geometry of $M$ then follows.

\begin{theorem} \label{T1} Let $(M, g)$ be a compact manifold with boundary and suppose
\begin{itemize}
\item $Ric_g\geq (n-1)g$;
\item $\partial M $ is isometric to $\mathbb{S}^{k-1}(\sin \theta)\times (N, g_N)$, where $\theta \in(0, \frac{\pi}{2})$ and
$N$ is an $(n-k)$ dimensional closed manifold with metric $g_N$;
\item the second fundamental form $h$ on $\partial M$ satisfies $h(w,w)\geq \cot \theta |w|^2$, $\forall w$ tangent to $\mathbb{S}^{k-1}$;
\item $H\geq (k-1)\cot \theta-(n-k) \tan \theta \geq 0$.
\end{itemize}
Then $M$ is isometric to the doubly warped product $dr^2+\sin^2(r) g_{\mathbb{S}^{k-1}(1)}+ \frac{\cos^2(r)}{\cos^2 \theta} g_N, r\in [0, \theta]$
and necessarily $Ric_{g_N}\geq\frac{(n-k-1)}{\cos^2 \theta} g_N$ in case $n-k\geq 2$.
\end{theorem}

Choose $N$ to be isometric to $\mathbb{S}^{n-k}(\cos \theta)$, we get
\begin{corollary} \label{C1}Let $(M, g)$ be a compact manifold with boundary, and suppose
\begin{itemize}
\item $Ric_g\geq (n-1)g$;
\item $\partial M $ is isometric to $\mathbb{S}^{k-1}(\sin \theta)\times \mathbb{S}^{n-k}(\cos \theta)$ for some $\theta\in(0, \frac{\pi}{2})$;
\item the second fundamental form $h$ satisfies $h(w,w)\geq \cot \theta |w|^2$, $\forall w$ tangent to $\mathbb{S}^{k-1}$;
\item $H\geq (k-1)\cot \theta-(n-k) \tan \theta \geq 0$.
\end{itemize} Then $(M,g)$ is isometric to $dr^2+\sin^2 r g_{\mathbb{S}^{k-1}}+ \cos^2 r  g_{\mathbb{S}^{n-k}}$, $r\in[0,\theta]$.
This is exactly the spherical domain bounded by a generalized clifford torus
$\mathbb{S}^{k-1}(\sin \theta)\times \mathbb{S}^{n-k}(\cos \theta)\subset \mathbb{S}^n$
whose boundary has nonnegative mean curvature with respect to the outward unit normal.
\end{corollary}

\begin{remark}
This corollary exhibits a new type of spherical region where Hang-Wang type rigidity holds. Its boundary is a generalized clifford torus. It is interesting to note that clifford tori are
isoparametric hypersurfaces with two distinct principal curvatures in sphere. One might explore rigidity for other spherical regions bounded by
isoparametric hypersurfaces. It may also be related to overdetermined problems in sphere. We refer the reader to the recent survey ~\cite{G} for more discussions.
We also remark that Miao and Wang ~\cite{MW} have obtained various interesting rigidity results on manifolds with boundary under Ricci curvature lower bound.
Corollary ~\ref{C1} provides a slight improvement of Theorem 1.5 in ~\cite{MW}, in terms of the assumption on the second fundamental form.
\end{remark}

Analogously, we obtain
\begin{theorem}\label{T2} Let $(M, g)$ be a compact manifold with boundary, and suppose
\begin{itemize}
\item $Ric_g\geq -(n-1)g$;
\item $\partial M $ is isometric to $\mathbb{S}^{k-1}(\sinh \theta)\times (N, g_N)$ ($\theta>0$) and
$N$ is an $(n-k)$ dimensional closed manifold with metric $g_N$;
\item the second fundamental form satisfies $h(w,w)\geq \coth \theta |w|^2$, $\forall w$ tangent to $\mathbb{S}^{k-1}$;
\item $H\geq (k-1)\coth \theta+(n-k) \tanh \theta$.
\end{itemize}
Then $M$ is isometric to the doubly warped product $dr^2+\sinh^2(r) g_{\mathbb{S}^{k-1}(1)}+ \frac{\cosh^2(r)}{\cosh^2 \theta} g_N, r\in [0, \theta]$
and necessarily $Ric_{g_N}\geq-\frac{(n-k-1)}{\cosh^2 \theta} g_N$, if $n-k\geq 2$.
\end{theorem}

\begin{theorem} \label{T3}Let $(M, g)$ be a compact manifold with boundary, and suppose
\begin{itemize}
\item $Ric_g\geq 0$;
\item $\partial M $ is isometric to $\mathbb{S}^{k-1}( \theta)\times (N, g_N)$ ($\theta>0$) and
$N$ is an $(n-k)$ dimensional closed manifold with metric $g_N$;
\item the second fundamental form satisfies $h(w,w)\geq  \frac{1}{\theta} |w|^2$, $\forall w$ tangent to $\mathbb{S}^{k-1}$;
\item $H\geq \frac{k-1}{\theta}$.
\end{itemize}
Then $M$ is isometric to $D^k(\theta) \times N$ with the product metric, where $D^k(\theta)$  is the Euclidean ball of radius $\theta$ and necessarily $Ric_{g_N}\geq 0$, if $n-k\geq 2$.
\end{theorem}
\begin{remark}
Notice that when $N$ is a one-dimensional closed manifold, namely a circle, then $M$ in Theorem~\ref{T2} is in fact a hyperbolic manifold as the sectional curvature
of $dr^2+\sinh^2(r) g_{\mathbb{S}^{n-2}(1)}+ \frac{\cosh^2(r)}{\cosh^2 \theta} g_N, r\in [0, \theta]$ is $\equiv-1$.
$M$ in Theorem~\ref{T3} is a flat manifold $D^{n-1}(\theta)\times \mathbb{S}^1$.
\end{remark}

In the next section, we first present the detailed proof of Theorem~\ref{T1}, then indicate necessary changes for proofs of Theorem~\ref{T2} and Theorem~\ref{T3}.

\section{Proof of the main theorem}
\begin{proof}[Proof of Theorem ~\ref{T1}]
The proof follows closely with the proof in ~\cite{HW}. First let us recall Reilly's result~\cite{R}.
Let $(M,g)$ be a compact manifold with boundary and $Ric_g\geq (n-1)g$. Suppose $\partial M$ is
mean convex, i.e., the mean curvature $H\geq 0$, then the first eigenvalue $\lambda_1$ of the Laplacian
operator with the Dirichlet boundary condition satisfies $\lambda_1\geq n$. Moreover the equality holds if and only if
$(M,g)$ is isometric to the standard hemisphere.

In view of assumptions on the boundary, we have $\lambda_1>n$. Thus there exists a unique solution to
\begin{align} \label{equ:u}
\left\{
  \begin{array}{ll}
    \Delta u+nu=0, & \text{in $M$}; \\
    u=f, &\text{on $\partial M$},
  \end{array}
\right.
\end{align}
for any $f \in C^{\infty}(\partial M)$.

Denote by $\bar{g}$ the induced metric of $M$ on $\partial M$, and by $\bar{\nabla}$, $\bar{\Delta}$ the induced
operators with respect to $\bar{g}$ on $\partial M$.

Since $\partial M$ is isometric to $\mathbb{S}^{k-1}(\sin \theta)\times N$, we could choose $f$ to be one of the coordinate functions on $\mathbb{S}^{k-1}$ that does not depend on the second factor.
By direct computations, we have
\begin{align} \label{1}
\bar{\Delta} f+(k-1)\csc^2 \theta f&=0 \quad \text{on $\partial M$},  \\
 \label{2}
\sin^2 \theta  |\bar{\nabla}f|^2+f^2&=\sin^2 \theta  \quad \text{on $\partial M$}.
\end{align}

Set $\frac{\partial u}{\partial \nu}=:\varphi$ and $u^2+|\nabla u|^2=:\phi$.
By the Bochner formula, $Ric\geq (n-1)g$ and (\ref{equ:u}), we have
\begin{align} \notag
\Delta \phi&= \Delta (u^2+|\nabla u|^2 ) \\ \notag
 &=2\left( u \Delta u+ |\nabla u|^2 + |\nabla^2 u|^2 + \nabla \Delta u \cdot \nabla u+ Ric(\nabla u, \nabla u)\right) \\ \notag
 &\geq 2\left( -nu^2 +|\nabla u|^2+ \frac{(\Delta u)^2}{n} -n |\nabla u|^2+ (n-1) |\nabla u|^2\right) \\   \notag
 &=0.
\end{align}

Our goal is to show $\phi$ is constant. Suppose not, then there exists $p
\in \partial M$ such that
\[
\phi(p)=\max _{x\in M} \phi(x) \quad \text{and} \quad \frac{\partial \phi}{\partial \nu}(p)>0.
\]
The latter is due to the Hopf lemma.

We compute $\frac{\partial \phi}{\partial \nu}$ as follows:
\begin{align} \label{11}
\frac{1}{2}\frac{\partial \phi}{\partial \nu}&= u \frac{\partial u}{\partial \nu}+D^2u(\nu, Du)\\ \notag
&=\varphi(f+D^2u(\nu, \nu))+ D^2 u(\nu, \bar{\nabla} f) \\ \notag
&=\varphi(f+D^2u(\nu, \nu))+ \nabla_{\bar{\nabla} f} \nabla u \cdot \nu \\ \notag
 &= \varphi(f+D^2u(\nu, \nu))+ \bar{\nabla} f\cdot \bar{\nabla} \varphi- h(\bar{\nabla} f, \bar{\nabla} f).
\end{align}

Set $A=\varphi(f+D^2u(\nu, \nu))$ and $B=\bar{\nabla} f\cdot \bar{\nabla} \varphi- h(\bar{\nabla} f, \bar{\nabla} f)$.

On $\partial M$, we have
\[
0=\Delta u+n u= \bar{\Delta} f+H \varphi+D^2u (\nu, \nu)+nf.
\]
Plugging (\ref{1}) in the above, it follows
\begin{align}  \label{4}
A&= \varphi(-H\varphi+(k-1) \cot^2 \theta f-(n-k)f) \\ \notag
 &\leq ((k-1)\cot \theta-(n-k) \tan \theta)\varphi [\cot \theta  f-\varphi],
\end{align}
where we have used the fact $H\geq (k-1)\cot \theta-(n-k) \tan \theta\geq0$.

By (\ref{2}), $\phi|_{\partial M}= 1+\varphi^2-\cot^2 \theta f^2$.
Since $p$ is maximal for $\phi$ and $f$ takes value $0$ somewhere, it follows that $\varphi^2(p)\geq \cot^2 \theta  f(p)^2$ and
\begin{align} \label{5}
0=\bar{\nabla }\phi (p)=2\varphi(p) \bar{\nabla}\varphi(p) -2 \cot^2 \theta f(p) \bar{\nabla} f(p).
\end{align}

There are two cases:
\begin{itemize}
  \item
 Case 1: $\varphi(p)\neq 0$. Since changing $f$ to $-f$ results a change $\varphi$ to $-\varphi$, but leaves
  $\phi$ unchanged, thus we can without loss of generality assume that $\varphi(p)>0$. Hence $-\varphi(p)\leq \cot \theta f(p)\leq \varphi(p)$,
  from which we obtain $A(p)\leq 0$.
  For $B$, we infer by (\ref{5}) that
  \begin{align} \label{6}
  B(p)&=\cot^2\theta \frac{f(p)}{\varphi(p)} |\bar{\nabla} f|^2- h(\bar{\nabla} f, \bar{\nabla} f) \\ \notag
 &\leq\cot \theta |\bar{\nabla} f|^2(\cot \theta \frac{f(p)}{\varphi(p)}-1)\leq 0.
  \end{align}
  Here we use the fact $\bar{\nabla} f$ is tangent to $\mathbb{S}^{k-1}$, where $h\geq \cot \theta $.
  This contradicts with $\frac{\partial \phi}{\partial \nu}(p)>0$.
  \item Case 2: $\varphi(p)=0$. In this case, $f(p)=0$ as well and $A(p)=0$. For $B$,
we use the fact $p$ is the maximum point of $\phi$, thus for any $X\in T(\partial M)$, we have
  \begin{align} \notag
  0\geq \bar{\nabla}^2_{X, X} \phi (p) &=\bar{\nabla}^2_{X, X}(1+\varphi^2-\cot^2 \theta f^2) (p) \\ \notag
                 &=2(X \cdot \bar{\nabla} \varphi)^2-2\cot^2 \theta (X\cdot \bar{\nabla}f)^2.
   \end{align}
Here we use the fact $p$ is minimum for both $\varphi^2$ and $f^2$.
   Letting $X=\bar{\nabla}f$, we get $\cot^2 \theta |\bar{\nabla}f|^2\geq \bar{\nabla}f\cdot \bar{\nabla} \varphi$,
which implies $B(p)\leq 0$. We again get a contradiction.
\end{itemize}

Therefore $\phi$ is constant on $M$, which implies $D^2u +ug=0$
in $M$. Moreover $\varphi^2-\cot^2 \theta f^2$ is also a constant on $\partial M$. Since
\[
0=\frac{\partial \phi}{\partial \nu}= A+B,
\]
combining the estimate (\ref{4}) and (\ref{6}), we have $\varphi=\cot \theta f$.
Plugging it back to (\ref{11}), we get
\begin{align} \notag
0&\equiv A+B =\varphi(f+D^2u(\nu, \nu))+\bar{\nabla} f\cdot \bar{\nabla} \varphi- h(\bar{\nabla} f, \bar{\nabla} f) \\ \notag
& =\varphi^2(-H+(k-1)\cot \theta-(n-k) \tan \theta)+ \cot \theta |\bar{\nabla} f|^2 -h(\bar{\nabla} f, \bar{\nabla} f).
\end{align}

In view of assumptions that $H\geq (k-1)\cot \theta-(n-k) \tan \theta$ and $h(\bar{\nabla} f, \bar{\nabla} f)\geq \cot \theta |\bar{\nabla}f|^2$, we infer that
$H\equiv (k-1)\cot \theta-(n-k) \tan \theta$ and $h(\bar{\nabla} f, \bar{\nabla} f)=\cot \theta |\bar{\nabla}f|^2$.

We summarize the known facts so far: we have found $u$ satisfying the Obata equation with the Robin boundary condition:
\begin{align} \label{7}
\left\{
  \begin{array}{ll}
    \nabla^2u +ug=0, & \text{in $M$} \\
    \frac{\partial u}{\partial \nu}-\cot \theta u=0, & \text{on $\partial M$}.
  \end{array}
\right.
\end{align}

An important fact is that there exists indeed a family of functions satisfying (\ref{7}) corresponding to the boundary value $f$. In particular,
there exist coordinate functions $f_1, \cdots f_k$ on $\mathbb{S}^{k-1}$ such that $\{\bar{\nabla} f_i\}_{i=1}^{k}$ span the tangent subspace along $\mathbb{S}^{k-1}$,
therefore we get
\[
h(v,v)\equiv \cot \theta, \quad \forall \text{ unit
vector tangent $v$ along $\mathbb{S}^{k-1}$}.
\]

To determine the metric structure of $M$, let $\partial M_t$ be the inward $t$-equidistance hypersurface of $\partial M$, where $M_{t}:=\{x\in M| dist(x, \partial M)\geq t\}$.
Thus for small $t$, say $t\in (0, t_0)$, $\partial M_t$ is diffeomorphic to $\partial M$.
The diffeomorphism $\Pi_t: \partial M \to \partial M_t$ is explicitly given by the inward normal exponential map, i.e. $\Pi_t(x)=\exp_x(-t\nu(x))$, for $x\in \partial M$.

Along a normal geodesic $\gamma_x(t)$ with $\gamma_x(0)=x$ and $\gamma_x'(0)=-\nu(x)$, $u$ satisfies
\[
u''\circ(\gamma_x(t))+u\circ(\gamma_x(t))=0.
\]
Note that $u'\circ(\gamma_x(0))=-\frac{\partial u}{\partial \nu}=-\cot \theta u(0)$ by (\ref{7}), it follows that
\begin{align} \label{12}
u\circ(\gamma_x(t))= u(0) \cos t - \cot\theta u(0) \sin t.
\end{align}
Then
\[
u'\circ(\gamma_x(t))=-u(0) \sin t-\cot \theta u(0) \cos t.
\]
Thus on $\partial M_t$,
\begin{align} \notag
\frac{\partial u}{\partial t}-\frac{\sin t+\cot \theta \cos t}{\cos t-\cot \theta \sin t} u=0.
\end{align}
For simplicity, let us set $\frac{\sin t+\cot \theta \cos t}{\cos t-\cot \theta \sin t}:=a(t)$. Hence the Obata equation (\ref{7}) holds on $M_{t}:=\{x\in M| dist(x, \partial M)\geq t\}$ subject to $a(t)$-Robin boundary condition
on $\partial M_t$:
\begin{align} \label{10}
    \frac{\partial u}{\partial \nu_t}-a(t) u=0.
\end{align}

Let $T_{t}\mathbb{S}^{k-1}$ and $T_{t} N$ be distributions of tangent vectors of $\partial M_t$ tangential to $\mathbb{S}^{k-1}$ and $N$ respectively.  We \textbf{claim} that $T_{t} \mathbb{S}^{k-1} \perp T_{t} N$ at each point of $\partial M_t$.

To this end, we can again choose the boundary value $f$ to be coordinate functions $f^{(1)}, \cdots f^{(k)}$ such that $\{ \bar{\nabla} f^{(i)}\}_{i=1}^{k}$ span the tangent subspace along $\mathbb{S}^{k-1}$.  Let $u^{(i)}$, $i=1, \cdots ,k$ be the corresponding solutions, thus they all satisfy (\ref{7}) and (\ref{10}).
For any $p\in \mathbb{S}^{k-1}$, by (\ref{12}) it follows that the restriction of $u^{(i)}$ on $\Pi_t(\{p\}\times N)$ is constant, thus
$\overline{\nabla_t} u^{(i)}$ is perpendicular to the tangent subspace along $\Pi_t(\{p\}\times N)$, which belongs to $T_{t}N$. Clearly $\overline{\nabla_t} u^{(i)}$ span the tangent
subspace along the $\mathbb{S}^{k-1}$ factor of $\partial M_t$, the claim thus follows.

Let $\{e_1, \cdots e_{n-1},e_n=\nu_t\}$ be an orthonormal frame, such that $\{e_1, \cdots, e_{k-1}\}$ forms an orthonormal frame
of $T_{t} \mathbb{S}^{k-1}$ and $\{e_k, \cdots, e_{n-1}\}$ forms an orthonormal frame of $T_{t} N$.

Using (\ref{7}) and (\ref{10}), for $i\neq n$, we have
\begin{align} \notag
0=\nabla^2_{e_i, e_n} u&= \nabla_{e_i}\nabla_{e_n} u-\nabla_{\nabla_{e_i} e_n} u \\ \notag
&=a(t) u_i- h(t)_{ij} u_j.
\end{align}
Hence $\overline{\nabla_t} u$ is an eigenvector of $h(t)$ with respect to the eigenvalue $a(t)$.  In view of above different choices of $u^{(i)}$, we have
\[
h(t)_{ii}=a(t), \quad \forall 1\leq i\leq k-1.
\]

Restricting to the boundary $\partial M_t$, (\ref{7}) reads as
\begin{align} \label{8}
\overline{\nabla_t}^2_{e_i, e_j} u+ a(t) u h(t)+ u\bar{g}(t)=0.
\end{align}

Let $\widetilde{\nabla_t}$ denote the induced Levi-Civita connection of $\bar{g}(t)$ on its submanifolds of the form
\[
\Pi_t(\mathbb{S}^{k-1}\times \{q\}) , \quad q\in N \quad \text{or} \quad \Pi_t(\{p\} \times N) ,\quad p \in \mathbb{S}^{k-1}.
\]

We have
\[
\overline{\nabla_t}^2_{e_i, e_j}u=\widetilde{\nabla_t}^2_{e_i, e_j}u- \overline{\nabla_t}_{(\overline{\nabla_t}_{e_i}e_j)^{\bot}} u, \quad 1\leq i,j\leq k-1.
\]
Since $(\overline{\nabla_t}_{e_i}e_j)^{\bot}$ is along $N$-factor direction, the last term vanishes. From (\ref{8}), we infer
\[
\widetilde{\nabla_t}^2_{e_i, e_j}u+ (1+a(t)^2) u\tilde{g}=0,
\] holds on $\mathbb{S}^{k-1}\times \{q\}$ for any $q\in N$. By Obata theorem~\cite{O}, it follows that $\Pi_t(\mathbb{S}^{k-1}\times \{q\})$ is isometric to
the round sphere of radius $\frac{1}{\sqrt{1+a(t)^2}}$.

Similarly
\[
\overline{\nabla_t}^2_{e_i, e_j}u=\widetilde{\nabla_t}^2_{e_i, e_j}u- \overline{\nabla_t}_{(\bar{\nabla_t}_{e_i}e_j)^{\bot}} u, \quad k\leq i,j\leq n-1.
\]
For fixed $p\in \mathbb{S}^{k-1}$, we may choose the boundary value $f$ to be the cosine of the distance function to $p$ (the coordinate function which attains maximum at $p$). Let $u$ be the corresponding solution, then $\overline{\nabla }u$ vanishes on $\Pi_t(p \times N)$.
Thus restricting (\ref{8}) on $\Pi_t(p\times N)$, we have
\[
a(t) u h(t)+ u \bar{g}(t)=0,
\]
which implies $h(t)_{jj}=-\frac{1}{a(t)}$, $\forall k\leq j\leq n-1$.

Therefore, we have shown that $h(t)$ on $\partial M_t$ has two distinct principal curvatures:
\begin{align}\notag
a(t) \quad & \text{of multiplicity $k-1$},\\ \notag
-\frac{1}{a(t)} \quad & \text{of multiplicity  $n-k$}.
\end{align}

Plugging this back to (\ref{8}) again, by the fact $u$'s restriction on $\Pi_t(\{p\}\times N)$ is constant and
\[
\overline{\nabla_t}^2_{e_i, e_j}u=\widetilde{\nabla_t}^2_{e_i, e_j}u- \overline{\nabla_t}_{(\overline{\nabla_t}_{e_i}e_j)^{\bot}} u, \quad k\leq i,j\leq n-1,
\] we indeed have
\[
\overline{\nabla_t}_{(\overline{\nabla_t}_{e_i}e_j)^{\bot}} u=0 \quad \forall u,
\] which
implies that $\Pi_t(\{p\}\times N) $ is totally geodesic in $\partial M_t$ for any $p \in \mathbb{S}^{k-1}$. Hence $\bar{g}(t)$ is a product metric. Writing $\bar{g}(t)= \alpha(t)_{\mathbb{S}^{k-1}}\oplus \beta(t)_{N}$, where $\alpha(t)$ and $\beta(t)$ are
two families of metrics on $\mathbb{S}^{k-1}$ and $N$ respectively. We thus have
\begin{align} \label{9}
\left(
  \begin{array}{cc}
    \frac{1}{2}  \alpha'(t) & 0 \\
    0 & \frac{1}{2} \beta'(t) \\
  \end{array}\right)=\left(
                \begin{array}{cc}
                  -a(t)\alpha(t) & 0 \\
                  0 & \frac{1}{a(t)} \beta(t) \\
                \end{array}
              \right).
\end{align}
The boundary condition means
\[
\alpha(0)=\sin^2 \theta g_{\mathbb{S}^{k-1}(1)}, \quad \beta(0)=g_{N}.
\]
Solving (\ref{9}), we obtain
\[
\alpha(t)= \sin^2 (\theta-t) g_{\mathbb{S}^{k-1}(1)}, \quad \beta(t)=\frac{\cos^2(\theta-t)}{\cos^2\theta} dg_{N}.
\]

So far we only determine the metric for small $ t\in(0,t_{0})$. However, an explicit calculation based on the Riccati's equation
 indicates there exists no focal point of $\partial M $ for $t=t_0$. Thus by continuity of the metric, $\partial M_{t_0}$ is isometric to $\mathbb{S}^{k-1}(\sin\theta-t_0)\times N$ with a
product metric. Therefore, we could continue the same argument as long as $a(t)$ is well defined, i.e, for $t\in[0,\theta)$.

Hence $g=dt^ 2+  \sin^2 (\theta-t) g_{\mathbb{S}^{k-1}(1)}+ \frac{\cos^2(\theta-t)}{\cos^2\theta} dg_{N}$ for $t\in[0,\theta)$.
After metric completion, which corresponds to a submanifold diffeomorphic to $N$ for $t=\theta$, it yields
a compact manifold with boundary, thus it is necessarily the whole $M$.
Letting $r=\theta-t$, then $(M,g)$ is isometric to the doubly warped product
\[
dr^2+\sin^2(r)  g_{\mathbb{S}^{k-1}(1)}+ \frac{\cos^2(r)}{\cos^2 \theta} g_N, \quad r\in [0, \theta].
\]

By direction computation of the Ricci curvature for this doubly warped product, we have
\begin{align} \notag
Ric(\partial_r, \partial_r)&=(n-1) , \\   \notag
Ric(X,X)&=(n-1),  \text{ $\forall X$ unit vector tangent to $\mathbb{S}^{k-1}$ factor,} \\  \notag
Ric(Y,Y)&= Ric_{g_N}(Y,Y) +k -(n-k-1)\frac{\sin^2 r}{\cos^2 r}, \text{$\forall Y$ unit vector tangent to $N$ factor} .
\end{align}
The assumption $Ric_g\geq (n-1)g$ implies that $Ric_{g_N}\geq \frac{(n-k-1)}{\cos^2 \theta} g_N$.
 \end{proof}

Using the same method, we can prove Theorem~\ref{T2} and Theorem~\ref{T3}.
\begin{proof}[Proof of Theorem~\ref{T2}]
Let $u$ be a solution of
\begin{align}  \notag
\left\{
  \begin{array}{ll}
    \Delta u-nu=0, & \text{in $M$}; \\
    u=f, &\text{on $\partial M$},
  \end{array}
\right.
\end{align}
for any $f \in C^{\infty}(\partial M)$. Note it is always solvable. Take $f$ to be a coordinate function on $\mathbb{S}^{k-1}$, which does not depend on the $N$ factor. This time we consider $\phi:=\frac{1}{2} (|\nabla u|^2- u^2)$.
A similar argument shows that $\phi=\text{const}$ and consequently $u$ satisfies
\begin{align}  \label{17}
\left\{
  \begin{array}{ll}
    D^2u -ug=0, & \text{in $M$} \\
    \frac{\partial u}{\partial \nu}-\coth \theta u=0, & \text{on $\partial M$}.
  \end{array}
\right.
\end{align}

Along a normal geodesic $\gamma_x(t)$ with $\gamma_x(0)=x\in \partial M$ and $\gamma_x'(0)=-\nu(x)$, $u$ satisfies
\[
u''\circ(\gamma_x(t))-u\circ(\gamma_x(t))=0.
\]
Note that $u'\circ(\gamma_x(0))=-\frac{\partial u}{\partial \nu}=-\coth \theta u(0)$ by (\ref{17}), it follows that
\[
u\circ(\gamma_x(t))= u(0) \cosh t - \coth\theta u(0 )\sinh t.
\]
Thus on the $t$-equidistance hypersurface $\partial M_t$,
\begin{align}  \notag
\frac{\partial u}{\partial t}-\frac{\coth \theta \cosh t-\sinh t}{\cosh t-\coth \theta \sinh t} u=0.
\end{align}
For simplicity, let $b(t)=\frac{\coth \theta \cosh t-\sinh t}{\cosh t-\coth \theta \sinh t}$. A similar argument shows that there are two distinct principal
curvatures: $b(t)$ whose eigenspace is the tangent subspace along $\mathbb{S}^{k-1}$ and $\frac{1}{b(t)}$ whose eigenspace is the tangent subspace
along $N$. Similarly $\bar{g}(t)=\alpha(t)_{\mathbb{S}^{k-1}}\oplus \beta(t)_{N}$ is a product metric on the $t$-equidistance hypersurface $\partial M_t$, and we get the metric ODE
\begin{align} \notag
\left(
  \begin{array}{cc}
    \frac{1}{2}  \alpha'(t) & 0 \\
    0 & \frac{1}{2} \beta'(t) \\
  \end{array}\right)=\left(
                \begin{array}{cc}
                  -b(t)\alpha(t) & 0 \\
                  0 & -\frac{1}{b(t)} \beta(t) \\
                \end{array}
              \right).
\end{align}
Solving it using the boundary condition, we get
\[
\alpha(t)=\sinh^2 (\theta-t) g_{\mathbb{S}^{k-1}(1)}, \quad \beta(t)=\cosh^2(\theta-t) g_N.
\]
The desired result thus follows.
\end{proof}

\begin{proof}[Proof of Theorem~\ref{T3}]
Let $u$ be a solution of
\begin{align} \notag
\left\{
  \begin{array}{ll}
    \Delta u=0, & \text{in $M$}; \\
    u=f, &\text{on $\partial M$},
  \end{array}
\right.
\end{align}
for any $f \in C^{\infty}(\partial M)$. Take $f$ to be a coordinate function of $\mathbb{S}^{k-1}$, which does not depend on the $N$ factor.
We consider $\phi:=\frac{1}{2} |\nabla u|^2$. A similar argument shows that $\phi=\text{const}$ and thus we have $u$ satisfies
\begin{align} \label{27}
\left\{
  \begin{array}{ll}
    D^2u =0, & \text{in $M$} \\
    \frac{\partial u}{\partial \nu}- \frac{1}{\theta} u=0, & \text{on $\partial M$}.
  \end{array}
\right.
\end{align}

Thus along a normal geodesic $\gamma_x(t)$ with $\gamma_x(0)=x\in \partial M$ and $\gamma_x'(0)=-\nu(x)$, $u$ satisfies
\[
u''\circ(\gamma_x(t))=0.
\]
Note that $u'\circ(\gamma_x(0))=-\frac{\partial u}{\partial \nu}=- \frac{1}{\theta} u(0)$ by (\ref{27}), it follows that
\[
u\circ(\gamma_x(t))= u(0)  -\frac{u(0)}{\theta} t.
\]
Thus on the $t$-equidistance hypersurface $\partial M_t$,
\begin{align}  \notag
\frac{\partial u}{\partial t}-\frac{1}{\theta-t} u=0.
\end{align}
For simplicity, let $c(t)=\frac{1}{\theta-t}$.

Again, there are
two distinct principal curvatures $\frac{1}{c(t)}$ and $0$, corresponding to the tangent subspace along $\mathbb{S}^{k-1}$ and $N$, respectively. The metric $\bar{g}(t)$ is of
a product metric, and we get the corresponding metric ODE
\begin{align} \notag
\left(
  \begin{array}{cc}
    \frac{1}{2}  \alpha'(t) & 0 \\
    0 & \frac{1}{2} \beta'(t) \\
  \end{array}\right)=\left(
                \begin{array}{cc}
                  -\frac{1}{c(t)}\alpha(t) & 0 \\
                  0 & 0  \\
                \end{array}
              \right).
\end{align}

In view of the boundary condition, we get
\[
\alpha(t)=(\theta-t)^2 g_{\mathbb{S}^{k-1}(1)}, \quad \beta(t)=g_N,
\]
which leads to the desired conclusion.
\end{proof}

\end{document}